\documentclass[12pt,reqno]{article}

\usepackage[usenames]{color}
\usepackage[colorlinks=true,
linkcolor=webgreen, filecolor=webbrown,
citecolor=webgreen]{hyperref}

\definecolor{webgreen}{rgb}{0,.5,0}
\definecolor{webbrown}{rgb}{.6,0,0}

\usepackage{amssymb}
\usepackage{mathtools}
\usepackage{graphicx}
\usepackage{amscd}
\usepackage{lscape}
\usepackage{tikz}
\usepackage{tikz-cd}
\usepackage{pgfplots}

\usetikzlibrary{matrix}
\usetikzlibrary{fit,shapes}
\usetikzlibrary{positioning, calc}
\tikzset{circle node/.style = {circle,inner sep=1pt,draw, fill=white},
        X node/.style = {fill=white, inner sep=1pt},
        dot node/.style = {circle, draw, inner sep=5pt}
        }
\usepackage{tkz-fct}

\usepackage{amsthm}
\newtheorem{theorem}{Theorem}
\newtheorem{lemma}[theorem]{Lemma}
\newtheorem{proposition}[theorem]{Proposition}
\newtheorem{corollary}[theorem]{Corollary}

\theoremstyle{definition}
\newtheorem{definition}[theorem]{Definition}
\newtheorem{example}[theorem]{Example}

\usepackage{float}

\usepackage{graphics,amsmath}
\usepackage{amsfonts}
\usepackage{latexsym}
\usepackage{epsf}

\newcommand{\seqnum}[1]{\href{http://oeis.org/#1}{\underline{#1}}}
\newcommand{\lpp}{(\mkern-3mu(}
\newcommand{\rpp}{)\mkern-3mu)}
\newcommand{\blpp}{\left(\mkern-7mu\left(}
\newcommand{\brpp}{\right)\mkern-7mu\right)}

\begin{document}

\begin{center}
\vskip 1cm{\LARGE\bf A new group in the Riordan family of matrix groups: the Sprugnoli group} \vskip 1cm \large
Paul Barry\\
School of Science\\
South East Technological University\\
Ireland\\
\href{mailto:pbarry@wit.ie}{\tt pbarry@wit.ie}
\end{center}
\vskip .2 in

\hspace*{40mm}\emph{Dedicated to the memory of Renzo Sprugnoli}

\begin{abstract} We define a group of lower-triangular matrices whose columns are defined by power series. This group can be seen as a generalization of the (ordinary) Riordan group and the double Riordan group. Elements of this group are defined by three power series. Sequence bisections and vertically stretched Riordan arrays play an important role in the formulation of this group. We give a production matrix characterization of this new group.  We also indicate how higher order groups can be defined, based on $n$-tuples of power series. We have chosen to name this group in memory of Renzo Sprugnoli, who was a pioneer in the application of the Riordan group to combinatorial problems as well as contributing to an understanding of the rich structure of Riordan arrays.
\end{abstract}

\textbf{Notation}: In this note, we shall use the usual notation $(g,f)$ to denote an element of the (ordinary) Riordan group \cite{book1, book2, SGWW}, the (non-standard) notation $\lpp g, f_1, f_2\rpp$ to denote an element of the double Riordan group \cite{double}, and $(g, f_1, f_2)$ to denote an element of the set of Sprugnoli arrays (defined hereafter). Note that integer sequences that occur in this note, where known, will be referred to by their On-Line Encyclopedia of Integer Sequences (OEIS) reference \cite{SL1, SL2}.

We recall that $\mathcal{F}_0 = \{ \sum_{n=0}^{\infty}a_n x^n\,|\, a_0 \ne 0\}$ and that $\mathcal{F}_1 = \{ \sum_{n=0}^{\infty}a_n x^n\,|\, a_0 =0, a_1 \ne 0\}$ and in general
$$\mathcal{F}_r = \{\sum_{n=r}^{\infty}a_n x^n\,|\, a_r \ne 0\}.$$ Elements of $\mathcal{F}_0$ are multiplicatively invertible, and elements of $\mathcal{F}_1$ are compositionally invertible, given suitable ground rings $R$ for $a_n \in R$.

The Riordan group is then the group of pairs $(g, f) \in \mathcal{F}_0 \times \mathcal{F}_1$ with the following product rule
\begin{equation}\label{product}(g(x), f(x))\cdot (u(x), v(x))=  (g(x)u(f(x)), v(f(x))\end{equation} and inverse
$$(g, f)^{-1}= \left(\frac{1}{g(\bar{f})}, \bar{f}(x)\right),$$ where $\bar{f}$ is the compositional inverse of $f \in \mathcal{F}_0$ (that is, $\bar{f}$ is the solution $u$ of the equation $f(u)=x$ for which $u(0)=0$). The identity of this group is $(1,x)$. To each element of this group we can associate in a unique way a lower-triangular matrix $(t_{n,k})$ with entries in the ground ring $R$ by means of
$$t_{n,k}=[x^n] g(x)f(x)^k,$$ where $[x^n]$ is the functional on $R[[x]]$ that extracts the coefficient of $x^n$. Under this correspondence, the product (\ref{product}) corresponds to ordinary matrix multiplication. The columns of $(g,f)$ have their generating functions given by the geometric series of power series
$$g, gf, gf^2, gf^3, gf^4, \ldots.$$ The bi-variate generating function of the Riordan array $(g(x), f(x))$ is given by $\frac{g(x)}{1-yf(x)}$.
\begin{example} The Riordan array $(g(x),f(x))=\left(\frac{1}{1-x}, \frac{x}{(1-x)^2}\right)$ begins
$$\left(\begin{array}{ccccccc}
1 & 0 & 0 & 0 & 0 & 0 & 0 \\
1 & 1 & 0 & 0 & 0 & 0 & 0 \\
1 & 3 & 1 & 0 & 0 & 0 & 0 \\
1 & 6 & 5 & 1 & 0 & 0 & 0 \\
1 & 10 & 15 & 7 & 1 & 0 & 0 \\
1 & 15 & 35 & 28 & 9 & 1 & 0 \\
1 & 21 & 70 & 84 & 45 & 11 & 1\\
\end{array}\right).$$
We have $t_{n,k}=\binom{n+k}{2k}$.
The generating function of this array is given by
$$\frac{\frac{1}{1-x}}{1-y\frac{x}{(1-x)^2}}=\frac{1-x}{1-x(y+2)+x^2}.$$
When $y=1$, we obtain the generating function $\frac{1-x}{1-3x+x^2}$ of the row sums of this matrix. Thus the row sums begin
$$1, 2, 5, 13, 34, 89, 233, 610, 1597, 4181, 10946,\ldots,$$ or $F_{2n+1}$ \seqnum{A122367}.
\end{example}  The product law follows from the following result, called the ``fundamental theorem of Riordan arrays'', which details how a Riordan array operates on a power series. We have
$$(g(x), f(x))\cdot h(x)= g(x)h(f(x)).$$ This is sometimes called a weighted composition.

A similar law holds for so-called ``vertically stretched'' Riordan arrays \cite{LI}. These are matrices corresponding to pairs of power series of the form $(g, xf(x))$ where as usual, $f(x) \in \mathcal{F}_1$, and so $xf(x) \in \mathcal{F}_2$. Again, we have a weighted composition action on power series given by
$$ (g(x), xf(x))\cdot h(x)= g(x)h(xf(x)).$$
The generating function of the stretched Riordan array $(g, xf(x))$  is given by
$$(g(x),xf(x))\cdot \frac{1}{1-xy}=\frac{g(x)}{1-yxf(x)}.$$ In particular, the generating function of the row sums of the stretched Riordan array $(g,xf)$ will have generating function $\frac{g(x)}{1-xf(x)}$.
\begin{example} The stretched Riordan array $\left(\frac{1}{1-x}, \frac{x^2}{1-x-x^2}\right)$ begins
$$\left(\begin{array}{ccccccccc}
1 & 0 & 0 & 0 & 0 & 0 & 0 & 0 & 0 \\
1 & 0 & 0 & 0 & 0 & 0 & 0 & 0 & 0 \\
1 & 1 & 0 & 0 & 0 & 0 & 0 & 0 & 0 \\
1 & 2 & 0 & 0 & 0 & 0 & 0 & 0 & 0 \\
1 & 4 & 1 & 0 & 0 & 0 & 0 & 0 & 0 \\
1 & 7 & 3 & 0 & 0 & 0 & 0 & 0 & 0 \\
1 & 12 & 8 & 1 & 0 & 0 & 0 & 0 & 0 \\
1 & 20 & 18 & 4 & 0 & 0 & 0 & 0 & 0 \\
1 & 33 & 38 & 13 & 1 & 0 & 0 & 0 & 0
\end{array}\right).$$ We have, for instance,
$$\left(\begin{array}{ccccccccc}
1 & 0 & 0 & 0 & 0 & 0 & 0 & 0 & 0 \\
1 & 0 & 0 & 0 & 0 & 0 & 0 & 0 & 0 \\
1 & 1 & 0 & 0 & 0 & 0 & 0 & 0 & 0 \\
1 & 2 & 0 & 0 & 0 & 0 & 0 & 0 & 0 \\
1 & 4 & 1 & 0 & 0 & 0 & 0 & 0 & 0 \\
1 & 7 & 3 & 0 & 0 & 0 & 0 & 0 & 0 \\
1 & 12 & 8 & 1 & 0 & 0 & 0 & 0 & 0 \\
1 & 20 & 18 & 4 & 0 & 0 & 0 & 0 & 0 \\
1 & 33 & 38 & 13 & 1 & 0 & 0 & 0 & 0
\end{array}\right)\left( \begin{array}{c}1\\1\\2\\3\\5\\8\\13\\21\\34\end{array}\right)=\left( \begin{array}{c}1\\1\\2\\3\\7\\14\\32\\69\\159\end{array}\right),$$ where the last sequence has generating function
$$\left(\frac{1}{1-x},\frac{x^2}{1-x-x^2}\right)\cdot \frac{1}{1-x-x^2}=\frac{\frac{1}{1-x}}{1-\frac{x^2}{1-x-x^2}-(\frac{x^2}{1-x-x^2})^2},$$ which is equal to
$$\frac{(1-x-x^2)^2}{(1-x)(1-2x+3x^3+x^4)}.$$
\end{example}
The \emph{double Riordan group} \cite{double} is a generalization of the checkerboard subgroup of the Riordan group \cite{book1, book2, SGWW}. It has as its elements $3$-tuples of power series $\lpp g,f_1,f_2 \rpp$ where
$g \in \mathbf{R}[[x^2]]$, $g_0 \ne 0$, and $f_1, f_2 \in x\mathbf{R}[[x^2]]$, $(f_i)_1 \ne 0$, where $\mathbf{R}$ is an appropriate ring, often taken to be $\mathbb{Z}$, or the fields $\mathbb{R}$ or $\mathbb{C}$. The element $\lpp g, f_1, f_2 \rpp$ of the double Riordan group is then represented by the lower-triangular matrix $\left(a_{n,k}\right)$ whose $(n,k)$-th element is given by
\begin{equation} a_{n,k}=[x^n] g(x)f_1(x)^{\lfloor \frac{k+1}{2} \rfloor} f_2(x)^{\lfloor \frac{k}{2} \rfloor}.\end{equation}
Thus the columns of this matrix have generating functions given by  $$g, gf_1, gf_1f_2, gf_1^2f_2, gf_1^2f_2^2,\ldots.$$
The identity of this group is $(1,x,x)$. The product rule for this group is given by first letting $h(x)=\sqrt{f_1f_2}$, and then we have
\begin{equation}\label{dproduct} \lpp g, f_1, f_2\rpp\cdot \lpp G, F_1, F_2 \rpp=\lpp gG(h), \frac{f_1}{h}F_1(h), \frac{f_2}{h}F_2(h) \rpp.\end{equation} Again, under the association of the triple $\lpp g,f_1,f_2 \rpp$ with the matrix $(a_{n,k})$ given by (\ref{dproduct}), the product in the group coincides with ordinary matrix multiplication. With regard to inverses in the group,
we have
$$\lpp g, f_1, f_2\rpp^{-1}=\lpp \frac{1}{g(\bar{h})}, \frac{x\bar{h}}{f_1(\bar{h})}, \frac{x\bar{h}}{f_2(\bar{h})}\rpp,$$ where $\bar{h}$ is the compositional inverse of $h$. Note that
$h(x) \in \mathcal{F}_1$, since both $f_1, f_2 \in \mathcal{F}_1$ and so $f_1(x)f_2(x) \in \mathcal{F}_2 \Longrightarrow h(x)=\sqrt{f_1(x)f_2(x)} \in \mathcal{F}_1$.
\begin{definition} The \emph{Sprugnoli set} is the set with elements $(g, f_1, f_2)$ with $g(x) \in \mathcal{F}_0$, $f_1(x) \in \mathcal{F}_1$, and $f_2(x) \in \mathcal{F}_1$ and
$f_2 \in xR[[x^2]]$ (thus $f_2$ is an odd power series). The element $(g, f_1, f_2)$ of this group has the matrix representation
$$t_{n,k} = [x^n] g(x)f_1(x)^{k \bmod 2} (xf_2(x))^{\lfloor \frac{k}{2} \rfloor}.$$
\end{definition}
Thus we have
$$
t_{n,k} =
\begin{cases}
  [x^n]g(x)(xf_2(x))^m, & k=2m, \\
  [x^n]g(x)f_1(x)(xf_2(x))^m,  & k=2m+1.
\end{cases}
$$
The columns of this matrix then have generating functions given by the following products of generating functions
\begin{align*}g(x), g(x)f_1(x),&\, g(x)(xf_2(x)), g(x)f_1(x)(xf_2(x)), g(x)(xf_2(x))^2,\\
& g(x)f_1(x)(xf_2(x))^2, g(x)(xf_2(x))^3, g(x)f_1(x)(xf_2(x))^3,\ldots\end{align*} We can represent this sequence by the following schema.
$$
\begin{array}{cccccccc}
g &g & g & g & g & g & g & g\\
1 &f_1 & 1 & f_1 & 1 & f_1 & 1 & f_1 \\
1 &1 & f_2 & f_2 & f_2^2 & f_2^2 & f_2^3 & f_2^3 \\
1 & 1 & x & x & x^2 & x^2 & x^3 & x^3 \\
- & - & - & - & - & - & - & -\\
0 & 1 & 2 & 3 & 4 & 5 & 6 & 7 \\
\end{array}
$$ Here, the final row gives the order of the product of the terms above it, showing that this associated matrix is lower-triangular. We can represent this array as the sum of two matrices as follows. The first matrix is the matrix whose columns are generated by the power series
$$g(x), 0, g(x)(xf_2(x)), 0, g(x)(xf_2(x))^2, \ldots,$$ which is a horizontal aeration of the stretched Riordan array $(g(x), xf_2(x))$, and the second matrix is the matrix whose columns are generated by the power series
$$0,g(x)f_1(x),0, g(x)f_1(x)(xf_2(x)), 0, g(x)f_1(x)(xf_2(x))^2,0,\ldots,$$ which is the horizontal aeration of the stretched Riordan array $(g(x)f_1(x), xf_2(x))$ (but note that this matrix starts with a zero row, due to the fact that $f_1 \in \mathcal{F}_1$).\

\begin{example} We consider the Sprugnoli array $\left(\frac{1}{1-x}, \frac{x(1+x)}{1-x}, \frac{x}{1-x^2}\right)$. This array begins
$$\left(\begin{array}{ccccccccc}
1 & 0 & 0 & 0 & 0 & 0 & 0 & 0 & 0 \\
1 & 1 & 0 & 0 & 0 & 0 & 0 & 0 & 0 \\
1 & 3 & 1 & 0 & 0 & 0 & 0 & 0 & 0 \\
1 & 5 & 1 & 1 & 0 & 0 & 0 & 0 & 0 \\
1 & 7 & 2 & 3 & 1 & 0 & 0 & 0 & 0 \\
1 & 9 & 2 & 6 & 1 & 1 & 0 & 0 & 0 \\
1 & 11 & 3 & 10 & 3 & 3 & 1 & 0 & 0 \\
1 & 13 & 3 & 15 & 3 & 7 & 1 & 1 & 0 \\
1 & 15 & 4 & 21 & 6 & 13 & 4 & 3 & 1
\end{array}\right).$$ This is the sum of the matrices
$$\left(\begin{array}{ccccccccc}
1 & 0 & 0 & 0 & 0 & 0 & 0 & 0 & 0 \\
1 & 0 & 0 & 0 & 0 & 0 & 0 & 0 & 0 \\
1 & 0 & 1 & 0 & 0 & 0 & 0 & 0 & 0 \\
1 & 0 & 1 & 0 & 0 & 0 & 0 & 0 & 0 \\
1 & 0 & 2 & 0 & 1 & 0 & 0 & 0 & 0 \\
1 & 0 & 2 & 0 & 1 & 0 & 0 & 0 & 0 \\
1 & 0 & 3 & 0 & 3 & 0 & 1 & 0 & 0 \\
1 & 0 & 3 & 0 & 3 & 0 & 1 & 0 & 0 \\
1 & 0 & 4 & 0 & 6 & 0 & 4 & 0 & 1
\end{array}\right)+
\left(\begin{array}{ccccccccc}
0 & 0 & 0 & 0 & 0 & 0 & 0 & 0 & 0 \\
0 & 1 & 0 & 0 & 0 & 0 & 0 & 0 & 0 \\
0 & 3 & 0 & 0 & 0 & 0 & 0 & 0 & 0 \\
0 & 5 & 0 & 1 & 0 & 0 & 0 & 0 & 0 \\
0 & 7 & 0 & 3 & 0 & 0 & 0 & 0 & 0 \\
0 & 9 & 0 & 6 & 0 & 1 & 0 & 0 & 0 \\
0 & 11 & 0 & 10 & 0 & 3 & 0 & 0 & 0 \\
0 & 13 & 0 & 15 & 0 & 7 & 0 & 1 & 0 \\
0 & 15 & 0 & 21 & 0 & 13 & 0 & 3 & 0
\end{array}\right).$$
The first is a horizontal aeration of the stretched Riordan array $\left(\frac{1}{1-x},\frac{x}{1-x^2}\right)$, while the second one
is a horizontal aeration of the stretched Riordan array $\left(\frac{1}{1-x} \frac{x(1+x)}{1-x},\frac{x}{1-x^2}\right)=\left(\frac{x(1+x)}{(1-x)^2},\frac{x}{1-x^2}\right)$ with an extra initial column of zeros.
\end{example}
A special element of this set is given by $(1,x,x)$. Its columns are generated by the sequence $1,x,x^2,x^3,\ldots$ and so its matrix representation is  given by the usual identity matrix.

Elements of this set will be referred to as Sprugnoli arrays, in reference to their matrix representation.
\begin{proposition}
The bi-variate generating function of the Sprugnoli array $(g, f_1, f_2)$ is given by
$$\frac{g(x)}{1-y^2xf_2(x)}+\frac{yg(x)f_1(x)}{1-y^2xf_2(x)}=\frac{g(x)(1+yf_1(x)}{1-y^2xf_2(x)}.$$
\end{proposition}
\begin{proof} The array is given by the sum of the horizontal aeration of the stretched Riordan array $(g(x), xf_2(x))$ and the horizontal aeration of the stretched Riordan array $(g(x)f_1(x), xf_2(x))$. Interpreting this in terms of generating functions gives the result.
\end{proof}
\begin{corollary} The row sums of the Sprugnoli array $(g, f_1, f_2)$ have generating function $\frac{g(x)(1+f_1(x))}{1-xf_2(x)}$.
\end{corollary}
\begin{proof} This results by setting $y=1$ in the generating function of the array.
\end{proof}
\begin{corollary} The diagonal sums of the Sprugnoli array $(g, f_1, f_2)$ have generating function $\frac{g(x)(1+xf_1(x))}{1-x^2f_2(x)}$.
\end{corollary}
\begin{proof} This results by setting $y=x$ in the generating function of the array.
\end{proof}

We wish to show that this set is a group. Thus we need first to state what the product in the group is. For this, we need to show how the elements of this set operate on power series. As with the Riordan group, this will then be translated in matrix terms into how the associated matrix operates (by matrix multiplication) on the vector whose elements are the coefficients of the given power series. This will give us the ``fundamental theorem of Sprugnoli arrays''.
\section{The ``fundamental theorem'' of Sprugnoli arrays}
The following example illustrates the result that follows.
\begin{example} We consider the Sprugnoli array $\left(\frac{1}{1-x}, \frac{x(1+x)}{1-x}, \frac{x}{1-x^2}\right)$ acting on the positive Fibonacci numbers. We have
$$\left(\begin{array}{ccccccccc}
1 & 0 & 0 & 0 & 0 & 0 & 0 & 0 & 0 \\
1 & 1 & 0 & 0 & 0 & 0 & 0 & 0 & 0 \\
1 & 3 & 1 & 0 & 0 & 0 & 0 & 0 & 0 \\
1 & 5 & 1 & 1 & 0 & 0 & 0 & 0 & 0 \\
1 & 7 & 2 & 3 & 1 & 0 & 0 & 0 & 0 \\
1 & 9 & 2 & 6 & 1 & 1 & 0 & 0 & 0 \\
1 & 11 & 3 & 10 & 3 & 3 & 1 & 0 & 0 \\
1 & 13 & 3 & 15 & 3 & 7 & 1 & 1 & 0 \\
1 & 15 & 4 & 21 & 6 & 13 & 4 & 3 & 1
\end{array}\right) \left(\begin{array}{c}1\\1\\2\\3\\5\\8\\13\\21\\34\end{array}\right)=$$
$$\left(\begin{array}{ccccccccc}
1 & 0 & 0 & 0 & 0 & 0 & 0 & 0 & 0 \\
1 & 0 & 0 & 0 & 0 & 0 & 0 & 0 & 0 \\
1 & 0 & 1 & 0 & 0 & 0 & 0 & 0 & 0 \\
1 & 0 & 1 & 0 & 0 & 0 & 0 & 0 & 0 \\
1 & 0 & 2 & 0 & 1 & 0 & 0 & 0 & 0 \\
1 & 0 & 2 & 0 & 1 & 0 & 0 & 0 & 0 \\
1 & 0 & 3 & 0 & 3 & 0 & 1 & 0 & 0 \\
1 & 0 & 3 & 0 & 3 & 0 & 1 & 0 & 0 \\
1 & 0 & 4 & 0 & 6 & 0 & 4 & 0 & 1
\end{array}\right)\left(\begin{array}{c}1\\0\\2\\0\\5\\0\\13\\0\\34\end{array}\right)+
\left(\begin{array}{ccccccccc}
0 & 0 & 0 & 0 & 0 & 0 & 0 & 0 & 0 \\
0 & 1 & 0 & 0 & 0 & 0 & 0 & 0 & 0 \\
0 & 3 & 0 & 0 & 0 & 0 & 0 & 0 & 0 \\
0 & 5 & 0 & 1 & 0 & 0 & 0 & 0 & 0 \\
0 & 7 & 0 & 3 & 0 & 0 & 0 & 0 & 0 \\
0 & 9 & 0 & 6 & 0 & 1 & 0 & 0 & 0 \\
0 & 11 & 0 & 10 & 0 & 3 & 0 & 0 & 0 \\
0 & 13 & 0 & 15 & 0 & 7 & 0 & 1 & 0 \\
0 & 15 & 0 & 21 & 0 & 13 & 0 & 3 & 0
\end{array}\right)\left(\begin{array}{c}0\\1\\0\\3\\0\\8\\0\\21\\0\end{array}\right)=$$
$$\left(\begin{array}{ccccccccc}
1 & 0 & 0 & 0 & 0 & 0 & 0 & 0 & 0 \\
1 & 0 & 0 & 0 & 0 & 0 & 0 & 0 & 0 \\
1 & 1 & 0 & 0 & 0 & 0 & 0 & 0 & 0 \\
1 & 1 & 0 & 0 & 0 & 0 & 0 & 0 & 0 \\
1 & 2 & 1 & 0 & 0 & 0 & 0 & 0 & 0 \\
1 & 2 & 1 & 0 & 0 & 0 & 0 & 0 & 0 \\
1 & 3 & 3 & 1 & 0 & 0 & 0 & 0 & 0 \\
1 & 3 & 3 & 1 & 0 & 0 & 0 & 0 & 0 \\
1 & 4 & 6 & 4 & 1 & 0 & 0 & 0 & 0
\end{array}\right)\left(\begin{array}{c}1\\ 2\\5\\13\\34\\89\\233\\610\\1597\end{array}\right)+
\left(\begin{array}{ccccccccc}
0 & 0 & 0 & 0 & 0 & 0 & 0 & 0 & 0 \\
1 & 0 & 0 & 0 & 0 & 0 & 0 & 0 & 0 \\
3 & 0 & 0 & 0 & 0 & 0 & 0 & 0 & 0 \\
5 & 1 & 0 & 0 & 0 & 0 & 0 & 0 & 0 \\
7 & 3 & 0 & 0 & 0 & 0 & 0 & 0 & 0 \\
9 & 6 & 1 & 0 & 0 & 0 & 0 & 0 & 0 \\
11 & 10 & 3 & 0 & 0 & 0 & 0 & 0 & 0 \\
13 & 15 & 7 & 1 & 0 & 0 & 0 & 0 & 0 \\
15 & 21 & 13 & 3 & 0 & 0 & 0 & 0 & 0
\end{array}\right)\left(\begin{array}{c}1\\ 3\\8\\21\\55\\144\\377\\987\\2584\end{array}\right).$$
Here, the sequences $1,2,5,\ldots$ and $1,3,8,\ldots$ are the even and odd bisections (see Appendix) of the positive Fibonacci numbers. The resulting sequence
$$1, 2, 6, 11, 26, 45, 100, 170, 370,\ldots$$ is the expansion of
$$\left(\frac{1}{1-x}, \frac{x(1+x)}{1-x}, \frac{x}{1-x^2}\right)\cdot \frac{1}{1-x-x^2}=\frac{(1+x)(1+x+x^3)}{1-5x^2+5x^4}.$$
\end{example} The general result is as follows.
\begin{proposition} Let $(g, f_1, f_2)$ be a Sprugnoli array, and let $h(x)$ be a power series $h_n(x)=\sum_{n=0}^{\infty} a_n x^n$. We let
$$h^e(x)=\sum_{n=0}^{\infty} a_{2n}x^n \quad\text{and}\quad h^o(x)=\sum_{n=0}^{\infty} a_{2n+1}x^n$$ be the even and odd bisections of $h(x)$. Then we have
$$(g(x), f_1(x), f_2(x))\cdot h(x)=g(x)h^e(xf_2(x))+g(x)f_1(x)h^o(xf_2(x)).$$
\end{proposition}
\begin{proof} We partition the matrix representing $(g, f_1, f_2)$ as
$$(g,0,g(xf_2),0,g(xf_2)^2,0,g(xf_2)^3,0,\ldots) + (0, gf_1, 0, gf_1(xf_2), 0, gf_1(xf_2)0,\ldots).$$
These matrices operate, respectively, on $(a_0,0,a_2,0,a_4,0,\ldots)$ and $(0,a_1,0,a_3,0,a_5,0,\ldots)$.
In terms of Riordan arrays, this is equivalent, after compression, to the stretched Riordan array $(g, xf_2)$ operating on $h^e$, and the
stretched Riordan array $(gf_1, xf_2)$ operating on $h^o(x)$. The result follows from this.
\end{proof}
\begin{example} We have
\begin{align*}(1, x, x) \cdot h(x)&= 1.h^e(x^2)+1.x.h^o(x^2)\\
&=h^e(x^2)+xh^o(x^2)\\
&=h(x).\end{align*}
\end{example}
\section{The product rule}
We are now in a position to say what we mean by the product of two Sprugnoli arrays. We note that the product rule in the Riordan group may be written 
\begin{align*}(g(x), f(x))\cdot (u(x),v(x))&=(g(x)u(f(x)), v(f(x))\\
&=\left((g(x),f(x))\cdot u, \frac{(g(x), f(x))\cdot u(x)v(x)}{(g(x),f(x))\cdot u(x)}\right).\end{align*}
This is so because 
\begin{align*}
(g(x), f(x))\cdot u(x) &= g(x)u(f(x)),\\
\frac{(g(x), f(x))\cdot u(x)v(x)}{(g(x),f(x))\cdot u(x)}&=\frac{g(x)u(f(x))v(f(x))}{g(x)u(f(x))}=v(f(x)).\end{align*}
Note that $u(x)$ and $u(x)v(x)$ are the generating functions of the first two columns of $(u(x), v(x))$. 
We can generalize this to give the following definition, using the first three columns of the second array.
\begin{definition} We define the product of two Sprugnoli arrays $(g, f_1, f_2)$ and $(u, v_1, v_2)$ as follows.
\begin{equation}\label{prod}(g, f_1, f_2) \cdot (u, v_1, v_2)= \left((g,f_1,f_2)\cdot u, \frac{(g,f_1,f_2)\cdot u v_1}{(g,f_1,f_2)\cdot u}, \frac{1}{x}.\frac{(g,f_1,f_2)\cdot uxv_2}{ (g,f_1,f_2)\cdot u}\right).\end{equation}
\end{definition}
To show that this is well defined and again a Sprugnoli array, we must show that
\begin{itemize}
\item $(g,f_1,f_2)\cdot u \in \mathcal{F}_0$,
\item $\frac{(g,f_1,f_2)\cdot u v_1}{(g,f_1,f_2)\cdot u} \in \mathcal{F}_1$,
\item $\frac{1}{x}.\frac{(g,f_1,f_2)\cdot uxv_2}{ (g,f_1,f_2)\cdot u} \in \mathcal{F}_1$ and is odd.
\end{itemize}
We have
$$(g,f_1,f_2)\cdot u = gu^e(xf_2)+gf_1u^o(xf_2).$$
Since $g(x) \in \mathcal{F}_0$ and $u^e(x) \in \mathcal{F}_0$, we obtain that $(g,f_1,f_2)\cdot u \in \mathcal{F}_0$. In particular, $(g,f_1,f_2)\cdot u$ is invertible so we have $\frac{1}{(g,f_1,f_2)\cdot u} \in \mathcal{F}_1$. Thus we must now show that
$(g,f_1,f_2)\cdot u v_1 \in \mathcal{F}_1$. We have
$$(g,f_1,f_2)\cdot u v_1=g(x)(uv_1)^e(xf_2(x))+g(x)f_1(x)(uv_1)^o(xf_2(x)).$$ We let
$u(x)=\sum_{n \ge 0} a_nx^n$ and $v_1(x)=\sum_{n\ge0}b_nx^n$. Then $uv_1$ expands as
$$a_0b_0, a_0b_1+a_1b_0, a_0b_2+a_1b_1+a_2b_0,\ldots.$$
Now $b_0=0$, so $(uv_1)^e$ expands as $0,a_0b_2+a_1b_1+a_2b_0,\ldots=0,a_0b_2+a_1b_1,\ldots.$ Thus $(uv_1)^e(x)\in \mathcal{F}_1$.
Similarly, we have that $(uv_1)^o(x)$ expands as $a_0b_1+a_1b_0,\ldots=a_0b_1,\ldots$ and so $(uv_1)^o(x) \in \mathcal{F}_0$.
This implies that $g(x)(uv_1)^e(xf_2(x)) \in \mathcal{F}_1$, and that $g(x)f_1(x)(uv_1)^o(xf_2(x))\in \mathcal{F}_1$ (since $f_1(x) \in \mathcal{F}_1$). We conclude that $(g,f_1,f_2)\cdot u v_1 \in \mathcal{F}_1$, and hence so is $\frac{(g,f_1,f_2)\cdot u v_1}{(g,f_1,f_2)\cdot u}$.

Applying similar reasoning to the previous paragraph, we can show that $\frac{1}{x} (g,f_1,f_2)\cdot uxv_2 \in \mathcal{F}_1$, and hence we have
$\frac{1}{x}.\frac{(g,f_1,f_2)\cdot uxv_2}{ (g,f_1,f_2)\cdot u} \in \mathcal{F}_1$.

\begin{lemma} We have $$(u(x)(xv_2(x))^m)^e = u^e(x) (\sqrt{x}v_2(\sqrt{x}))^m.$$
\end{lemma}
\begin{proof}
We have
\begin{align*}
(u(x)(xv_2(x))^m)^e(x)&=\frac{u(\sqrt{x})(\sqrt{x}v_2(\sqrt(x)))^m+u(-\sqrt{x})(-\sqrt{x}v_2(-\sqrt(x)))^m}{2}\\
&=\frac{u(\sqrt{x})(\sqrt{x}v_2(\sqrt(x)))^m+u(-\sqrt{x})(\sqrt{x}v_2(\sqrt(x)))^m}{2} \quad\text{($v_2$ is odd)}\\
&=\frac{u(\sqrt{x})+u(-\sqrt{x})}{2} (\sqrt{x}v_2(\sqrt{x}))^m\\
&=u^e(x) (\sqrt{x}v_2(\sqrt{x}))^m.\end{align*}
\end{proof}
In a similar way we have the following.
\begin{lemma}
We have $$(u(x)v_1(x)(xv_2(x))^m)^e = (uv_1)^e(x) (\sqrt{x}v_2(\sqrt{x}))^m.$$
\end{lemma}
\begin{proposition} We have
$$(g, f_1, f_2)\cdot u(x)(xv_2(x))^m = ((g,f_1,f_2)\cdot u)(\sqrt{xf_2(x)} v_2(\sqrt{xf_2(x)}))^m.$$
\end{proposition}
\begin{proof}
We have
\begin{align*}
(g, f_1, f_2)\cdot u(x)(xv_2(x))^m&=(g,xf_2)\cdot u^e(x) (\sqrt{x}v_2(\sqrt{x}))^m+(gf_1,xf_2)\cdot u^o(x)(\sqrt{x}v_2(\sqrt{x}))^m\\
&=(g(x)u^e(xf_2)(\sqrt{xf_2}v_2(\sqrt{xf_2}))^m+g(x)f_1(x)u^o(xf_2)(\sqrt{xf_2}v_2(\sqrt{xf_2}))^m\\
&=(g(x)u^e(xf_2)+g(x)f_1(x)u^o(xf_2))(\sqrt{xf_2}v_2(\sqrt{xf_2}))^m\\
&=((g,f_1,f_2)\cdot u)(\sqrt{xf_2}v_2(\sqrt{xf_2}))^m.\end{align*}
\end{proof}
\begin{corollary} We have $$\frac{1}{x} \frac{(g,f_1,f_2) \cdot uxv_2}{(g, f_1, f_2)\cdot u}=\frac{1}{x} \sqrt{xf_2}v_2(\sqrt{xf_2}).$$
\end{corollary}
\begin{proof} We set $m=1$ in the above result.
\end{proof}
\begin{proposition} $\frac{1}{x}.\frac{(g,f_1,f_2)\cdot uxv_2}{ (g,f_1,f_2)\cdot u}$ is odd.
\end{proposition}
\begin{proof}
We have
$$\frac{1}{x} \frac{(g,f_1,f_2) \cdot uxv_2}{(g, f_1, f_2)\cdot u}=\frac{1}{x} \sqrt{xf_2}v_2(\sqrt{xf_2}).$$
Changing $x$ to $-x$, we get
$$\frac{1}{-x} \sqrt{-xf_2(-x)}v_2(\sqrt{-xf_2(-x)})=\frac{1}{-x}\sqrt{xf_2(x)}v_2(\sqrt{xf_2(x)})=-\frac{1}{x}\sqrt{xf_2(x)}v_2(\sqrt{xf_2(x)})$$ since $f_2(x)$ is odd.
\end{proof}
We are now in a position to conclude that the product (\ref{prod}) is well defined and returns a Sprugnoli array. Thus the set of Sprugnoli arrays is closed under multiplication. Note that we now have
\begin{equation}\label{prod2}(g, f_1, f_2) \cdot (u, v_1, v_2)= \left((g,f_1,f_2)\cdot u, \frac{(g,f_1,f_2)\cdot u v_1}{(g,f_1,f_2)\cdot u}, \frac{1}{x} \sqrt{xf_2}v_2(\sqrt{xf_2})\right).\end{equation}
Furthermore, the product rule (\ref{prod}), when we look to the matrix representation of Sprugnoli arrays, coincides with matrix multiplication. To see this, we take the case of $k=2m$. In terms of matrix multiplication, the $k$-th column of the product of $(g, f_1, f_2)$ and $(u, v_1, v_2)$ will have generating function given by $(g, f_1, f_2)$ acting on the generating function of the $k$-th column of $(u, v_1, v_2)$. This is given by $$(g, f_1, f_2)\cdot (u(x)(xv_2(x))^m=((g,f_1, f_2)\cdot u)(\sqrt{xf_2}v_2(\sqrt{xf_2}))^m.$$ Now the $k$-th column of the product (\ref{prod2})
is given by $((g,f_1,f_2)\cdot u)(x \frac{1}{x}\sqrt{xf_2}v_2(\sqrt{xf_2}))^m$. Thus these coincide. A similar argument can be made for $k$ odd. We conclude that in the matrix representation, the product (\ref{prod2}) is given by ordinary matrix multiplication. Associativity of the product now follows from this.

\section{The inverse of a Sprugnoli array}
Our quest now is to define the inverse of a Sprugnoli array. For this, we first look at connections to the Riordan group and the double Riordan group.
\begin{lemma} If $f_1=\sqrt{xf_2}$, then the Sprugnoli array $(g,f_1, f_2)$ is given by the Riordan array $(g, f_1)=(g, \sqrt{xf_2})$.
\end{lemma}
\begin{proof} The columns of $(g, \sqrt{xf_2}, f_2)$ are generated by the sequence
$$g, g \sqrt{xf_2}, g (xf_2), g \sqrt{xf_2} g(xf_2), g (xf_2)^2, \ldots.$$ Clearly, this is
$$g, g f_1, g f_1^2, gf_1^3, gf_1^4,\ldots,$$ as required.
\end{proof}
Note that from $f_1=\sqrt{xf_2}$, we get $xf_2=f_1^2$ and $f_2=\frac{1}{x}f_1^2$.  Thus we can write
$\left(g, \sqrt{xf_2}, f_2)=(g, \sqrt{xf_2}\right).$ Now we have the Riordan array inverse
$$\left(g, \sqrt{x f_2}\right)^{-1}=\left(\frac{1}{g(\overline{\sqrt{x f_2}})}, \overline{\sqrt{x f_2(x)}}\right).$$ Re-interpreting this in terms of Sprugnoli arrays, we then have
$$(g, \sqrt{x f_2}, f_2)^{-1}=\left(\frac{1}{g(\overline{\sqrt{x f_2}})}, \overline{\sqrt{x f_2}}, \frac{1}{x}\left(\overline{\sqrt{x f_2}}\right)^2\right).$$
We now look at a situation where there is a link between Sprugnoli arrays and double Riordan arrays.
\begin{lemma} If $g(x) \in R[[x^2]], f_1 \in xR[[x^2]]$, and we set $F_1=f_1, F_2=\frac{xf_2}{f_1}$, then
$$(g, f_1, f_2)=\lpp g, F_1, F_2 \rpp.$$
\end{lemma}
\begin{proof} The double Riordan array $\lpp g, F_1, F_2 \rpp$ is defined by its column defining sequence
$$g, gF_1, gF_1 F_2, gF_1^2F_2, gF_1^2F_2^2,\ldots.$$ But by definition, this is the same as
$$g, gf_1, g(xf_2), gf_1(xf_2), g(xf_2)^2,\ldots.$$ By matching terms we obtain the result.
\end{proof}
Thus under the conditions of the last lemma we have
$$(g, f_1, f_2)^{-1}=\lpp g, F_1, F_2 \rpp^{-1}=\blpp\frac{1}{g(\bar{H})}, \frac{x\bar{H}}{F_1(\bar{H})}, \frac{x \bar{H}}{F_2(\bar{H})}\brpp,$$
where $H=\sqrt{F_1 F_2}=\sqrt{f_1 \frac{xf_2}{f_1}}=\sqrt{xf_2}$, and so $\bar{H}=\overline{\sqrt{xf_2}}$. Then under the conditions of the lemma, we have
\begin{align*} (g, f_1, f_2)^{-1}&=\blpp\frac{1}{g(\bar{H})}, \frac{x\bar{H}}{F_1(\bar{H})}, \frac{x \bar{H}}{F_2(\bar{H})}\brpp\\
&=\left(\frac{1}{g(\bar{H})}, \frac{x\bar{H}}{f_1(\bar{H})}, \frac{x\bar{H}}{F_2(\bar{H})}\cdot \frac{x\bar{H}}{F_1(\bar{H})}\cdot\frac{1}{x}\right)\\
&=\left(\frac{1}{g(\bar{H})}, \frac{x\bar{H}}{f_1(\bar{H})}, \frac{x\bar{H}^2}{(xf_2)(\bar{H})}\right),\quad\text{where $F_1F_2=xf_2$}\\
&=\left(\frac{1}{g(\bar{H})}, \frac{x\bar{H}}{f_1(\bar{H})}, \frac{x\left(\overline{\sqrt{xf_2}}\right)^2}{(xf_2)(\overline{\sqrt{xf_2}})}\right)\\
&=\left(\frac{1}{g(\bar{H})}, \frac{x\bar{H}}{f_1(\bar{H})}, \frac{1}{x} \left(\overline{\sqrt{xf_2}}\right)^2\right).\end{align*}

We have now seen two special cases in which the inverse of $(g, f_1, f_2)$ satisfies
$$(g, f_1, f_2)^{-1}=\left(w(x),s_1(x), \frac{1}{x} \left(\overline{\sqrt{xf_2}}\right)^2\right),$$ where $w(x)$ and $s_1(x)$ are as yet undetermined elements of the inverse.
To progress our search for an inverse in the general case, we have the following lemma concerning a canonical factorization of Sprugnoli arrays.
\begin{lemma}
We have
$$(g,x,x)\cdot (1, f_1, f_2)=(g,f_1,f_2).$$
\end{lemma}
\begin{proof}
We have
$$(g,x,x)\cdot 1=g.1^e(x^2)+g.1^o(x^2)=g,$$ since $1^e=1$ and $1^o=0$.
Now $$(g,x,x) \cdot (1, f_1, f_2)=\left((g,x,x)\cdot 1, \frac{(g,x,x)\cdot f_1}{(g,x,x)\cdot 1}, \frac{1}{x}\frac{(g,x,x)\cdot xf_2}{(g,x,x)\cdot 1}\right).$$
We have
\begin{align*} (g,x,x)\cdot f_1&=g.f_1^e(x^2)+gxf_1^o(x^2)\\
&=g(f_1^e(x^2)+xf_1^o(x^2))\\
&=gf_1.\end{align*}
Then $$\frac{(g,x,x)\cdot f_1}{(g,x,x)\cdot 1}=\frac{gf_1}{g}=f_1.$$
The third element of $(g,x,x)\cdot (1,f_1, f_2)$ is given by
$$\frac{1}{x} \frac{(g,x,x)\cdot xf_2}{(g,x,x)\cdot 1}=\frac{1}{x}\frac{(g,x,x)\cdot xf_2}{g}.$$ This is equal to
\begin{align*}\frac{1}{x} \frac{g.(xf_2)^e(x^2)+xg.(xf_2)^o(x^2)}{g}&=\frac{1}{x} ((xf_2)^e(x^2)+x(xf_2)^o(x^2))\\
&=\frac{1}{x} (x(f_2)^e(x^2)+x^2(f_2)^o(x^2))=(f_2)^e(x^2)+x(f_2)^o(x^2)\\
&=f_2(x).\end{align*}
\end{proof}
\begin{corollary} $$(g, f_1, f_2)^{-1}=(1, f_1,f_2)^{-1}\cdot (g, x, x)^{-1}.$$
\end{corollary}
\begin{proof} This is so since the product in the set of Sprugnoli arrays is given by matrix multiplication in the matrix representation. Thus the inverse of a product is given in the usual way.
\end{proof}
Now $(g,x,x)=(g,x)$ so we have $$(g,x,x)^{-1}=(g,x)^{-1}=\left(\frac{1}{g},x\right)=\left(\frac{1}{g},x,x\right).$$
Thus we have
\begin{equation}\label{inv0}(g, f_1, f_2)^{-1}=(1, f_1, f_2)^{-1} \cdot \left(\frac{1}{g}, x,x\right).\end{equation}
We now let $(1,r_1, r_2)=(1, f_1, f_2)^{-1}$. We have
$$(1, r_1, r_2) \cdot (1, f_1, f_2)=(1, x,x)$$ which gives us that
$$(1, r_1, r_2)\cdot f_1=x,$$ or
$$1. f_1^e(xr_2)+1.r_1f_1^o(xr_2)=x.$$ Solving for $r_1$, we obtain
$$r_1=\frac{x-f_1^e(xr_2)}{f_1^o(xr_2)}.$$
We are now in a position to say what the inverse of $(g,f_1, f_2)$ in the Sprugnoli group is.
\begin{definition} The inverse of the Sprugnoli group element $(g, f_1, f_2)$ is given by
$$(g, f_1, f_2)^{-1}=\left(1, \frac{x-f_1^e(xr_2)}{f_1^o(xr_2)}, r_2\right)\cdot \left(\frac{1}{g}, x, x\right)$$ where
$$r_2 = \frac{1}{x} \left(\overline{\sqrt{xf_2}}\right)^2.$$
\end{definition}
Letting $(g, f_1, f_2)^{-1}=(w(x), s_1(x), s_2(x))$ we then have
\begin{align*}
w(x)&=(1, r_1, r_2)\cdot \frac{1}{g}\\
&=1.\left(\frac{1}{g}\right)^e (xr_2)+\frac{x-f_1^e(xr_2)}{f_1^o(xr_2)}\left(\frac{1}{g}\right)^o(xr_2).\end{align*}
\begin{align*}
s_1(x)&=\frac{(1,r_1,r_2)(x/g)}{(1,r_1,r_2)(1/g)}\\
&=\frac{1.\left(x/g\right)^e(xr_2)+1.r_1\left(x/g\right)^o(xr_2)}{1.\left(1/g\right)^e(xr_2)+1.r_1\left(1/g\right)^o(xr_2)}\\
&=\frac{(x\left(1/g\right)^o)(xr_2)+r_1\left(1/g\right)^e(xr_2)}{\left(1/g\right)^e(xr_2)+r_1\left(1/g\right)^o(xr_2)}\\
&=\frac{xr_2(x).\left(1/g\right)^o(xr_2)+r_1\left(1/g\right)^e(xr_2)}{w(x)}.
\end{align*}
\begin{align*}
s_2(x)&=\frac{1}{x}\frac{(1,r_1,r_2)(x^2/g)}{(1,r_1,r_2)(x^2/g)}\\
&=\frac{1}{x}\frac{1.\left(x^2/g\right)^e(xr_2)+1.r_1\left(x^2/g\right)^o(xr_2)}{1.\left(1/g\right)^e(xr_2)+1.r_1\left(1/g\right)^o(xr_2)}\\
&=\frac{1}{x}\frac{((x)(\frac{1}{g})^e(xr_2)+r_1((x)(\frac{1}{g})^o(xr_2))}{(\frac{1}{g})^e(xr_2)+r_1(\frac{1}{g})^o(xr_2)}\\
&=\frac{1}{x} xr_2(x)\\
&=r_2(x).\end{align*}
\section{The Sprugnoli matrix group}
We have the following result.
\begin{theorem} The set of Sprugnoli arrays is a group.
\end{theorem}
\begin{proof} The set of Sprugnoli arrays is closed under the product (\ref{prod}) that we have defined. Moreover, this product, in the matrix representation, coincides with matrix multiplication. The product is therefore associative. We have also defined an identity element $(1,x,x)$ and have demonstrated that each element has an inverse. Thus the set is a group.
\end{proof}
To honour the memory of Renzo Sprugnoli, we call this group the \emph{Sprugnoli group}.

We now study some simple subgroups of the Sprugnoli group.
\begin{proposition} The set of Sprugnoli arrays of the form $(g,x,x)$ for $g \in \mathcal{F}_1$ is a subgroup of the Sprugnoli group.
\end{proposition}
\begin{proof} The identity for this subgroup is $(1,x,x)$. The inverse of $(g,x,x)$ is $\left(\frac{1}{g},x,x\right)$. We have
$$(g,x,x) \cdot (h,x,x)=\left((g,x,x)\cdot h, \frac{(g,x,x)\cdot hx}{(g,x,x)\cdot h}, \frac{1}{x}\sqrt{x^2} x(\sqrt{x^2})\right).$$
We have $\frac{1}{x}\sqrt{x^2} x(\sqrt{x^2})=x$.
\begin{align*}
(g,x,x)\cdot hx&=(g,x^2)(hx)^e+(gx,x^2)(hx)^o(x)\\
&=(g,x^2)xh^o(x)+(gx,x^2)h^e(x) \quad\text{see Appendix}\\
&=g(x)x^2h^o(x^2)+gxh^e(x^2)\\
&=x(gh^e(x^2)+gxh^o(x^2))\\
&=x(g,x,x)\cdot h.\end{align*}
Thus $(g,x,x)\cdot (h,x,x)=((g,x,x)h,x,x)$.
\end{proof}
\begin{proposition} The set of Sprugnoli arrays of the form $(1, f_1, f_2)$ is a subgroup of the Sprugnoli group.
\end{proposition}
\begin{proof} We have $$(1,f_1,f_2)\cdot (1,v_1,v_2)=((1,f_1,f_2)\cdot 1, \ldots, \ldots)=(1, \ldots, \ldots).$$
Also, $(1,f_1,f_2)^{-1}=(1,r_1,r_2)$ as calculated above.
\end{proof}
\begin{proposition} The set of Sprugnoli arrays of the form $(1,x,f_2)$ is a subgroup of the Sprugnoli group.
\end{proposition}
\begin{proof}
We have $(1,x,f_2)^{-1}=\left(1,x, \frac{1}{x}(\overline{\sqrt{xf_2}})^2\right)$. Also,
\begin{align*}
(1,x,f_2)\cdot (1,x,v_2)&=\left((1,x,f_2)\cdot 1,\frac{(1,x,f_2)\cdot 1.x}{(1,x,f_2)\cdot 1},\ldots\right)\\
&=(1, (1,xf_2)\cdot x^e(x)+(x,xf_2)x^o(x), \ldots)\\
&=(1, 0+x,\ldots).\end{align*}
\end{proof}
\begin{proposition} The set of Sprugnoli arrays of the form $(1,f_1,x)$ is a subgroup of the Sprugnoli group.
\end{proposition}
\begin{proof}
We have
$$(1,f_1,x)\cdot(1,v_1,x)=\left((1,f_1,x)\cdot 1, \frac{(1,f_1,x)\cdot v_1}{(1,f_1,x)\cdot 1}, \frac{1}{x}\frac{(1,f_1,x)\cdot x^2}{(1,f_1,x)\cdot 1}\right).$$
Now $(1,f_1,x)\cdot1=1$, and
$$(1,f_1,x)\cdot x^2=(1,x^2)(x^2)^e(x)+(f_1,x^2)\cdot (x^2)^o(x)=(1,x^2)\cdot x=x^2.$$
Hence
$$\frac{1}{x}\frac{(1,f_1,x)\cdot x^2}{(1,f_1,x)\cdot 1}=x.$$
We have $$(1,f_1,x)^{-1}=\left(1, \frac{x-f_1^e(x^2)}{f_1^o(x^2)},x\right).$$
\end{proof}

\section{Examples of inverse calculations}
\begin{example} We consider the Sprugnoli array \seqnum{A051159} which is $\left(\frac{1}{1-x}, \frac{x}{1+x}, \frac{x}{1-x^2}\right)$.
We have $xf_2(x)=\frac{x^2}{1-x^2}$. We seek the reversion of $\sqrt{xf_2(x)}$. Thus we wish to solve the equation
$$\frac{u}{\sqrt{1-u^2}}=x$$ for the solution $u(x)$ with $u(0)=0$. We have
$$\frac{u^2}{1-u^2}=x^2 \Longrightarrow u^2=\frac{x^2}{1+x^2}.$$
Now for the inverse, we have $s_2=r_2=\frac{u^2}{x}=\frac{x}{1+x^2}$.
We are now in a position to solve for $r_1$. We have
\begin{align*}
r_1&=\frac{x-f_1^e(xr_2)}{f_1^o(xr_2)}\\
&=\frac{x-\left(\frac{-x}{1-x}\right)\left(\frac{x^2}{1+x}\right)}{\left(\frac{1}{1-x}\right)\left(\frac{x^2}{1+x^2}\right)}\\
&=\frac{x-\frac{-x^2/(1+x^2)}{1-x^2/(1+x^2)}}{\frac{1}{1-x^2/(1+x^2)}}\\
&=x\left(1-\frac{x^2}{1+x^2}\right)+\frac{x^2}{1+x^2}\\
&=\frac{x(1+x)}{1+x^2}.\end{align*}
We have $\frac{1}{g}=1-x$ and so we have
$$\left(\frac{1}{g}\right)^e=1\quad\text{and}\quad\left(\frac{1}{g}\right)^o=-1.$$
Then
$$w(x)=(1,r_1,r_2)\cdot \frac{1}{g}=1-\frac{x(1+x)}{1+x^2}.1=1-\frac{x(1+x)}{1+x^2}=\frac{1-x}{1+x^2}.$$
Now $\frac{x}{g}=x(1-x)=x-x^2$, which gives us $\left(\frac{x}{g}\right)^e=-x$, and $\left(\frac{1}{g}\right)^o=1$.
Using these values and the equation
$$\frac{1.\left(x/g\right)^e(xr_2)+1.r_1\left(x/g\right)^o(xr_2)}{1.\left(1/g\right)^e(xr_2)+1.r_1\left(1/g\right)^o(xr_2)}$$ we find that
$$s_1(x)=\frac{x}{1-x}.$$
Thus we have
$$\left(\frac{1}{1-x}, \frac{x}{1+x}, \frac{x}{1-x^2}\right)^{-1}=\left(\frac{1-x}{1+x^2}, \frac{x}{1-x}, \frac{x}{1+x^2}\right).$$
We note that the rows of this matrix are palindromic (the matrix is ``Pascal-like'').
$$\left(\begin{array}{ccccccccc}1 & 0 & 0 & 0 & 0 & 0 & 0 & 0 & 0 \\
1 & 1 & 0 & 0 & 0 & 0 & 0 & 0 & 0 \\
1 & 0 & 1 & 0 & 0 & 0 & 0 & 0 & 0 \\
1 & 1 & 1 & 1 & 0 & 0 & 0 & 0 & 0 \\
1 & 0 & 2 & 0 & 1 & 0 & 0 & 0 & 0 \\
1 & 1 & 2 & 2 & 1 & 1 & 0 & 0 & 0 \\
1 & 0 & 3 & 0 & 3 & 0 & 1 & 0 & 0 \\
1 & 1 & 3 & 3 & 3 & 3 & 1 & 1 & 0 \\
1 & 0 & 4 & 0 & 6 & 0 & 4 & 0 & 1\end{array}\right).$$
\end{example}
\begin{example}
We take the example of the Sprugnoli array
$$S=\left(\frac{1+2x}{1-4x}, \frac{x(1+3x)}{1-2x}, \frac{x(1+x^2)}{1-x^2}\right).$$ We wish to find its inverse.
In order to find $s_2=r_2$, we have $s_2=\frac{u^2}{x}$ were $u$ satisfies the equation
$$u \sqrt{\frac{1+u^2}{1-u^2}}=x.$$ Squaring this equation, we get
$$u^4+(1+x^2)u^2-x^2=0$$ from which we deduce that
$$r_2=s_2=\frac{u^2}{x}=\frac{\sqrt{1+6x^2+x^4}-x^2-1}{2x}.$$
To calculate $r_1$, we must first get the even and odd bisections of $f_1(x)=\frac{x(1+3x)}{1-2x}$.
For this, we have
$$\frac{f_1(x)+f_1(-x)}{2}=\frac{5x^2}{1-4x^2},$$ and so we have
$$f_1^e(x)=\frac{5x}{1-4x}.$$
Then
$$\frac{f_1(x)-f_1(-x)}{2x}=\frac{1+6x^2}{1-4x^2},$$ and so we have
$$f_1^o=\frac{1+6x}{1-4x}.$$
Thus we get
$$r_1 = \frac{x-\left(\frac{5x}{1-4x}\right)\left(\frac{\sqrt{1+6x^2+x^4}-x^2-1}{2}\right)}{\left(\frac{1+6x}{1-4x}\right)\left(\frac{\sqrt{1+6x^2+x^4}-x^2-1}{2}\right)}.$$
This give us
$$r_1=\frac{5 (1+2x)\sqrt{1+6x^2+x^4}-46x^3-65x^2-5}{2(5+42x^2)}.$$
We now have the components of  $$(1, f_1, f_2)^{-1}=(1, r_1, r_2).$$
We must apply this array to $$\left(\frac{1}{g(x)}, x,x\right)=\left(\frac{1-4x}{1+2x},x,x\right).$$
We have
$$\left(\frac{1}{g}\right)^e = \frac{1+8x}{1-4x} \quad \text{and}\quad \left(\frac{1}{g}\right)^o=\frac{-6}{1-4x}.$$
We get
\begin{align*}w(x)&=(1,r_1,r_2)\cdot \frac{1}{g}\\
&=1.\left(\frac{1+8x}{1-4x}\right)(xr_2)+1.r_1\left(\frac{-6}{1-4x}\right)(xr_2).\end{align*}
Using the values for $r_1$ and $r_2$ already calculated, we find that
$$w(x)=\frac{3(5-30x+204x^2+72x^3)\sqrt{1+6x^2+x^4}+216x^5+1620x^4+54x^3+543x^2-60x+10}{(5-12x^2)(5+42x^2)}.$$
To calculate $s_1$, we first calculate the numerator of the required quotient, to obtain the generating function of the second column of the inverse.
This is given by
$$\left(\frac{-6x}{1-4x}\right)(xr_2)+r_1 \left(\frac{1+8x}{1-4x}\right)(xr_2).$$
Expanding this last expression, we obtain the generating function of the second column of the inverse to be
$$\frac{(24x^3-492x^2-10x-65)\sqrt{1+6x^2+x^4}-1320x^5-3684x^4+406x^4+137x^2+60x+65}{2(5-12x^2)(5+42x^2)}.$$
Then $s_1$ is this generating function divided by $w(x)$. This gives us that $s_1$ is equal to
$$\frac{(24x^3-492x^2-10x-65)\sqrt{1+6x^2+x^4}-1320x^5-3684x^4+406x^3+137x^2+60x+65}{2(3(5-30x+204x^2+72x^3)\sqrt{1+6x^2+x^4}+216x^5+1620x^4+54x^2-60x+10)}.$$
Finally we have $s_2(x)=r_2(x)$ to obtain the inverse $(w(x), s_1(x), s_2(x))$.
\end{example}
\section{Matrix elements, production matrix, and linear recurrences}
A defining feature of Riordan arrays is that they have a sequence characterization. For instance, in a Riordan array $M=(g,f)$, any matrix element $t_{n,k}$ not in the first column will obey a linear recurrence equation involving the elements in the row above it, starting with the element $t_{n-1,k-1}$ above and to the left of it. The same recurrence coefficients are valid for each such element, irrespective of the row it is in. This means that the production matrix $P$ of a Riordan array has a special simple structure, where $P=M^{-1}\overline{M}$, where $\overline{M}$ is the matrix $M$ with its top row removed. The elements in the first column of $M$ normally follow a distinct recurrence. This results in two sequences of recurrence coefficients, the $Z$ sequence (for the first column elements) and the $A$ sequence, for all the other elements. The production array will then have a structure as in the following matrix, where the obvious pattern continues.
$$\left(\begin{array}{rrrrrrr}
z_0 & a_0 & 0 & 0 & 0 &  0 & \cdots \\
z_1 & a_1 & a_0 & 0 & 0 & 0 & \cdots \\
z_2 & a_2 & a_1 &a_0 & 0 & 0 & \cdots \\
z_3 & a_3 & a_2 & a_1 & a_0 & 0 & \cdots \\
z_4 & a_4 & a_3 & a_2 & a_1 & a_0& \cdots \\
z_5 & a_5 & a_4 & a_3 & a_2 & a_1 & \cdots \\
\vdots & \vdots & \vdots & \vdots & \vdots & \vdots & \ddots
\end{array}\right).$$
For Sprugnoli arrays, a similar situation arises, but this time, we have two $A$-sequences (we designate them the $A$ sequence and the $B$-sequence), one for elements in even columns, and one for elements in odd columns. Thus the form of the production array for a Sprugnoli array is given by
$$\left(\begin{array}{rrrrrrr}
z_0 & a_0 & 0 & 0 & 0 &  0 & \cdots \\
z_1 & a_1 & b_0 & 0 & 0 & 0 & \cdots \\
z_2 & a_2 & b_1 &a_0 & 0 & 0 & \cdots \\
z_3 & a_3 & b_2 & a_1 & b_0 & 0 & \cdots \\
z_4 & a_4 & b_3 & a_2 & b_1 & a_0& \cdots \\
z_5 & a_5 & b_4 & a_3 & b_2 & a_1 & \cdots \\
\vdots & \vdots & \vdots & \vdots & \vdots & \vdots & \ddots
\end{array}\right),$$ giving it a ``striped'' appearance.
We let $A(x)$ be the generating function of the sequence $a_0, a_1, a_2,\ldots$, and let $B(x)$ be the generating function of the sequence $b_0, b_1, b_2,\ldots$. We have the following result.
\begin{proposition}
We have
\begin{align*} A(x)&=(1,r_1,r_2)\cdot \frac{f_1(x)}{x},\\
B(x)&=\frac{1}{x} (1,r_1,r_2)\cdot f_2(x).\end{align*}
\end{proposition}
\begin{proof}
$A(x)$ is the generating function obtained by applying the inverse array $(g, f_1, f_2)^{-1}=(1,r_1,r_2)\cdot \left(\frac{1}{g},x,x\right)$ to the second column of $\overline{(g,f_1, f_2)}$. This gives that
\begin{align*}A(x)&=(1,r_1,r_2)\cdot \left(\frac{1}{g}, x, x\right)\cdot g(x)\frac{f_1(x)}{x} \\
&=(1,r_1,r_2)\cdot \frac{1}{g(x)} g(x)\frac{f_1(x)}{x}\\
&=(1,r_1,r_2) \cdot  \frac{f_1(x)}{x}.\end{align*}
Similarly, by applying the inverse matrix to the third column of $\overline{(g,f_1, f_2)}$, we get
\begin{align*}xB(x)&=(1,r_1,r_2)\cdot \left(\frac{1}{g}, x, x\right) \cdot gf_2\\
&=(1,r_1,r_2)\cdot \frac{1}{g(x)} gf_2\\
&=(1,r_1,r_2) \cdot f_2.\end{align*}
Thus $B(x)=\frac{1}{x} (1,r_1, r_2)\cdot f_2$.
\end{proof}
We note that $Z(x)$, the generating function of the first column of the production matrix, is given by
$$Z(x)=\left(\frac{1-g_0/g(x)}{x}\right)^e(xr_2(x))+r_1(x)\left(\frac{1-g_0/g(x)}{x}\right)^o(xr_2(x)).$$
In the expression $B(x)=\frac{1}{x} (1,r_1, r_2)\cdot f_2$, we recall that $f_2(x)$ is an odd power series, and thus $f_2^e$ is zero. Hence we have
\begin{align*}xB(x)&=r_1(x)(f_2)^o(xr_2(x))\\
&=\frac{x-f_1^e(xr_2(x))}{f_1^o(xr_2(x))} (f_2)^o(xr_2(x))\\
&=(x-f_1^e(xr_2(x))\cdot \frac{f_2^o(xr_2(x))}{f_1^o(xr_2(x))}.\end{align*}
In order to see what the relationship between $A(x)$ and $B(x)$ is, we shall need the following lemma.
\begin{lemma}
We have
$$r_2(x)=\frac{x}{f_2^o(xr_2(x))}.$$
\end{lemma}
\begin{proof}
We have
$$f_2^o(x)=\frac{f_2(\sqrt{x})-f_2(-\sqrt{x})}{2\sqrt{x}}=\frac{f_2(\sqrt{x})+f_2(\sqrt{x})}{2\sqrt{x}}=\frac{f_2(\sqrt{x})}{\sqrt{x}},$$ where we have used the fact that $f_2(x)$ is odd.
Now $xr_2(x)=x.\frac{1}{x}\left(\overline{\sqrt{xf_2(x)}}\right)^2=\left(\overline{\sqrt{xf_2(x)}}\right)^2$, and so
$$f_2^o(xr_2(x))=f_2^o\left(\left(\overline{\sqrt{xf_2(x)}}\right)^2\right)=\frac{f_2(\overline{\sqrt{xf_2}})}{\overline{\sqrt{xf_2}}}.$$
Then
\begin{align*}
\frac{x}{f_2^o(xr_2(x))}&=\frac{x \overline{\sqrt{xf^2}}}{f_2(\overline{\sqrt{xf_2}})}\\
&=\frac{ x (\overline{\sqrt{xf_2}})^2}{\overline{\sqrt{xf_2}} f_2(\overline{\sqrt{xf_2}})}\\
&=\frac{ x (\overline{\sqrt{xf_2}})^2}{(xf_2)(\overline{\sqrt{xf_2}})}\\
&=\frac{ x (\overline{\sqrt{xf_2}})^2}{x^2}\\
&=\frac{1}{x} (\overline{\sqrt{xf_2}})^2\\
&=r_2(x)\end{align*}
\end{proof}
This leads us to the following characterization of $A(x)$ and $B(x)$ for Sprugnoli arrays.
\begin{proposition}
Let $A(x)$ and $xB(x)$ correspond to the generating functions of the second and third columns of the production array of the Sprugnoli array $(g, f_1, f_2)$. Then $A(x)+B(x)$ is an even power series.
\end{proposition}
\begin{proof}
We have
$$A(x)+B(x)=\frac{\frac{1}{x}(x-f_1^e(xr_2(x)))f_2^o(xr_2(x))}{f_1^o(xr_2(x))}+\frac{x-f_1^e(xr_2(x))}{f_1^o(xr_2(x))}\left(\frac{f_1(x)}{x}\right)^o(xr_2(x))+f_1^o(xr_2(x)).$$
Now $$\left(\frac{f_1(x)}{x}\right)^o=\frac{f_1^e}{x}.$$
Thus
$$A(x)+B(x)=\frac{\frac{1}{x}(x-f_1^e(xr_2(x)))f_2^o(xr_2(x))}{f_1^o(xr_2(x))}+\frac{x-f_1^e(xr_2(x))}{f_1^o(xr_2(x))}\frac{f_1^e(xr_2(x))}{xr_2(x)}+f_1^o(xr_2(x)).$$
Gathering terms, we find that
\begin{align*}A(x)+B(x)&=x \frac{f_1^e(xr_2(x))}{xr_2(x)}-\frac{1}{x}f_2^o(xr_2(x))f_1^e(xr_2(x))+\text{even terms in $xr_2(x)$}\\
&=f_1^e(xr_2(x))\left(\frac{x}{xr_2(x)}-\frac{1}{x}f_2^o(xr_2(x))\right)+\text{even terms}\\
&=0+\text{even terms}.\end{align*}

\end{proof}
\begin{example} We consider the Sprugnoli array $\left(\frac{1}{1-x-x^2}, \frac{x(1+x)}{1-x}, \frac{x}{1-x^2}\right)$. This matrix begins
$$\left(
\begin{array}{ccccccccc}
1 &0 & 0 & 0 & 0 & 0 & 0 & 0 & 0\\
1 & 1 & 0 & 0 & 0 & 0 & 0 & 0 & 0\\
2 & 3 & 1 & 0 & 0 & 0 & 0& 0 & 0 \\
3 & 6 & 1 & 1 & 0 & 0 & 0& 0 & 0 \\
5 & 11 & 3 & 3 & 1 & 0 & 0 & 0 & 0\\
8 & 19 & 4 & 7 & 1 & 1 & 0& 0 & 0 \\
13 & 32 & 8 & 14 & 4 & 3 & 1& 0 & 0 \\
21 & 53 & 12 & 26 & 5 & 8 & 1& 1 & 0 \\
34 & 87 & 21 & 46 & 12 & 17 & 5& 3 & 1 \\
\end{array}
\right).$$ The production matrix for this array is as follows.
$$\left(
\begin{array}{ccccccccc}
1  & 1 & 0 & 0 & 0 & 0 & 0 & 0 & 0\\
1  & 2 & 1 & 0 & 0 & 0 & 0 & 0 & 0\\
-2 & -2 & -2 & 1 & 0 & 0 & 0& 0 & 0 \\
-2 & -2 & -1 & 2 & 1 & 0 & 0& 0 & 0 \\
4  & 4 & 2 & -2 & -2 & 1 & 0 & 0 & 0\\
4  & 4 & 2 & -2 & -1 & 2 & 1& 0 & 0 \\
-8 & -8 & -4 & 4 & 2 & -2 & -2& 1 & 0 \\
-8 & -8 & -4 & -4 & 2 & -2 & -1& 2 & 1 \\
16 & 16 & 8 & -8 & -4 & 4 & 2& -2 & -2 \\
\end{array}
\right).$$
Thus we have
\begin{align*}
Z:&\quad1,1,-2,-2,4,4,-8,-8,16,\ldots\\
A:&\quad1,2,-2,-2,4,4,-8,-8,16,\ldots\\
B:&\quad1,-2,-1,2,2,-4,-4,8,8,\ldots \end{align*}
Taking the first column term $t_{6,0}=13$, we have (using dot-product notation)
$$13=[8,19,4,7,1,1,\ldots]\cdot [1,1,-2,-2,4,4,\ldots].$$
Taking the term $t_{7,1}=53$ in the second column, we have
$$53=[13,32,8,14,4,3,1,\ldots]\cdot [1,2,-2,-2,4,4,8,\ldots].$$
Note that the index of the column is odd, so we use the $A$ sequence as coefficients.
Now taking the term $t_{6,2}=8$, we have
$$8=[19,4,7,1,1,\ldots]\cdot [1,-2,-1,2,2,\ldots].$$
\end{example}
In general, we have
\begin{align*}t_{n,0}&=t_{n-1,0}z_0+t_{n-1,1}z_1+t_{n-1,2}z_2+ \cdots\\
t_{n,k}&=t_{n-1,k-1}a_0+t_{n-1,k}a_1+t_{n-1,k+1}a_2 + \cdots \quad\quad\text{for $k$ odd}\\
t_{n,k}&=t_{n-1,k-1}b_0+t_{n-1,k}b_1+t_{n-1,k+1}b_2 + \cdots \quad\quad\text{for $k$ even}.\end{align*}

\begin{example}
We consider that family of polynomials $P_n(x)$ defined by
$$
P_n(x) =
\begin{cases}
  (x+1)P_{n-1}(x)-P_{n-2}(x), & \text{if $n$ is even}, \\
  (x-1)P_{n-1}(x)-P_{n-2}(x),  & \text{if $n$ is odd},
\end{cases}
$$
with $P_n(x)=0$ for $n<0$ and $P_0(x)=1$. These polynomials begin
$$1, x - 1, x^2 - 2, x^3 - x^2 - 3x + 3, x^4 - 5x^2 + 5,\ldots.$$ The coefficient array of this sequence of polynomials is given by the Sprugnoli array
$$\left(\frac{1-x+x^2}{1+3x^2+x^4}, \frac{x}{1-x+x^2}, \frac{x}{1+3x^2+x^4}\right)$$ which begins
$$\left(
\begin{array}{cccccccccc}
1 & 0 & 0 & 0 & 0 & 0 & 0 & 0 & 0 \\
-1 & 1 & 0 & 0 & 0 & 0 & 0 & 0 & 0 \\
-2 & 0 & 1 & 0 & 0 & 0 & 0 & 0 & 0 \\
3 & -3 & -1 & 1 & 0 & 0 & 0 & 0 & 0 \\
5 & 0 & -5 & 0 & 1 & 0 & 0 & 0 & 0 \\
-8 & 8 & 6 & -6 & -1 & 1 & 0 & 0 & 0 \\
-13 & 0 & 19 & 0 & -8 & 0 & 1 & 0 & 0 \\
21 & -21 & -25 & 25 & 9 & -9 & -1 & 1 & 0 \\
34 & 0 & -65 & 0 & 42 & 0 & -11 & 0 & 1
\end{array}
\right).$$
We recognise a signed version of the Fibonacci numbers \seqnum{A000045} in the first column.
The inverse of this matrix begins
$$\left(
\begin{array}{cccccccccc}
1 & 0 & 0 & 0 & 0 & 0 & 0 & 0 & 0 \\
1 & 1 & 0 & 0 & 0 & 0 & 0 & 0 & 0 \\
2 & 0 & 1 & 0 & 0 & 0 & 0 & 0 & 0 \\
2 & 3 & 1 & 1 & 0 & 0 & 0 & 0 & 0 \\
5 & 0 & 5 & 0 & 1 & 0 & 0 & 0 & 0 \\
5 & 10 & 5 & 6 & 1 & 1 & 0 & 0 & 0 \\
15 & 0 & 21 & 0 & 8 & 0 & 1 & 0 & 0 \\
15 & 36 & 21 & 29 & 8 & 9 & 1 & 1 & 0 \\
51 & 0 & 86 & 0 & 46 & 0 & 11 & 0 & 1
\end{array}
\right).$$
and its production matrix  begins
$$\left(
\begin{array}{cccccccccc}
1 & 1 & 0 & 0 & 0 & 0 & 0 & 0 & 0 \\
1 & -1 & 1 & 0 & 0 & 0 & 0 & 0 & 0 \\
0 & 1 & 1 & 1 & 0 & 0 & 0 & 0 & 0 \\
0 & 0 & 1 & -1 & 1 & 0 & 0 & 0 & 0 \\
0 & 0 & 0 & 1 & 1 & 1 & 0 & 0 & 0 \\
0 & 0 & 0 & 0 & 1 & -1 & 1 & 0 & 0 \\
0 & 0 & 0 & 0 & 0 & 1 & 1 & 1 & 0 \\
0 & 0 & 0 & 0 & 0 & 0 & 1 & -1 & 1 \\
0 & 0 & 0 & 0 & 0 & 0 & 0 & 1 & 1
\end{array}
\right).$$
For this matrix, we have
\begin{align*}
g(x)&=\frac{1+x}{2x^2}\left(1-\sqrt{\frac{1-5x^2}{1-x^2}}\right)=\frac{1-x^2-\sqrt{(1-x^2)(1-5x^2)}}{2x^2(1-x)},\\
f_1(x)&=\frac{x(1-x)}{2x^2}\left(1-\sqrt{\frac{1-5x^2}{1-x^2}}\right)=\frac{1-x^2-\sqrt{(1-x^2)(1-5x^2)}}{2x(1+x)},\\
f_2(x)&=\frac{1}{2x^3}(1-3x^2-\sqrt{(1-x^2)(1-5x^2)}).\end{align*}
The first column of this inverse matrix, which begins
$$1, 1, 2, 2, 5, 5, 15, 15, 51, 51,\ldots,$$ or \seqnum{A055879} can be regarded as moments of the family of (generalized) orthogonal polynomials $P_n(x)$. This sequence is a doubling of the binomial transform of the Catalan numbers $C_n=\frac{1}{n+1}\binom{2n}{n}$ \seqnum{A000108}. The generating function of this moment sequence can be expressed as the Jacobi continued fraction
$$\cfrac{1}{1-x-\cfrac{x^2}{1+x-\cfrac{x^2}{1-x-\cfrac{x^2}{1+x-\cfrac{x^2}{1+x-\cdots}}}}}.$$
\end{example}

\section{Generalizations}
The above considerations point the way to further generalizations. For instance, a next Sprugnoli group of elements $(g, f_1, f_2, f_3)$ is based on $f_3 \in xR[[x^3]]$, and involves sequence trisections in the corresponding fundamental theorem, and cubic roots in the construction of the inverses. We call this group the $3$-Sprugnoli group. Thus the Riordan group is the $1$-Sprugnoli group, and the Sprugnoli group defined above is the $2$-Sprugnoli group. 

To define the ``fundamental theorem of the $3$-Sprugnoli group'', we need to calculate the $3$ trisections of power series. For a generating function $g(x)$, we can calculate its trisections $g^{(1)}(x)$, $g^{(2)}(x)$ and $g^{(3)}{(x)}$ using the $3 \times 3$ discrete Fourier transform as follows. We let $\omega=\exp(2\pi i/3)$. Then
$$\frac{1}{3}
\begin{pmatrix}
1 & 0 & 0\\
0 & x^{-1/3} & 0\\
0 & 0 & x^{-2/3}
\end{pmatrix}
\begin{pmatrix}
1 & 1 & 1 \\
1 & \omega^2 & \omega\\
1 & \omega & \omega^2
\end{pmatrix}
\begin{pmatrix}
g\!\left(x^{1/3}\right) \\
g\!\left(\omega x^{1/3}\right)\\
g\!\left(\omega^2 x^{1/3}\right)
\end{pmatrix}
=
\begin{pmatrix}
g^{(1)}(x) \\
g^{(2)}(x) \\
g^{(3)}(x)
\end{pmatrix}.$$
The fundamental theorem for this group, describing how an element $(g, f_1, f_2, f_3)$ acts on a power series $h(x)$, is given by 
$$(g,f_1,f_2,f_3)\cdot h(x)=g(x)h^{(1)}(x^2f_3(x))+g(x)f_1(x)h^{(2)}(x^2f_3(x))+g(x)f_1(x)f_2(x)h^{(3)}(x^2f_3(x)).$$ 
This allows us to define the multiplication in this group. We have 
\begin{scriptsize}
$$(g,f_1,f_2,f_3)\cdot (u,v_1, v_2, v_3)=\left((g,f_1,f_2,f_3)\cdot u, \frac{(g,f_1,f_2,f_3)\cdot uv_1}{(g,f_1,f_2,f_3)\cdot u},\frac{(g,f_1,f_2,f_3)\cdot uv_1v_2}{(g,f_1,f_2,f_3)\cdot uv_1}, \frac{1}{x^2}\frac{(g,f_1,f_2,f_3)\cdot ux^2v_3}{(g,f_1,f_2,f_3)\cdot u}\right).$$ 
\end{scriptsize}
The identity for this product is then given by $(1,x,x,x)$. In order to calculate the inverse of a group element, we look at the case of the element $(1, f_1, f_2, f_3)$ first. We let
$$(1, f_1, f_2, f_3)^{-1}=(1, r_1, r_2, r_3).$$ 
We have $$r_3(x)=\frac{1}{x^2}\left(\overline{\left(x^2f_3(x)\right)^{1/3}}\right)^3.$$
In order to calculate $r_1$ and $r_2$, we use the following matrix formulation, where $f_{12}(x)=f_1(x)f_2(x)$.
$$\begin{pmatrix}
f_1^{(2)}(x^2 r_3(x)) & f_1^{(3)}(x^2 r_3(x)) \\
f_{12}^{(2)}(x^2 r_3(x))/x & f_{12}^{(3)}(x^2 r_3(x))/x
\end{pmatrix}
\begin{pmatrix}
r_1(x) \\
r_1(x)r_2(x)
\end{pmatrix}
=
\begin{pmatrix}
x-f_1^{(1)}(x^2 r_3(x)) \\
x-f_{12}^{(1)}(x^2 r_3(x))/x
\end{pmatrix}.$$
This allows us to solve for $r_1(x)$ and $r_1(x)r_2(x)$. Thus we can find $r_2(x)$. Finally, in the general case, we use 
$$(g,f_1, f_2, f_3)^{-1}=(1, f_1, f_2, f_3)^{-1} \cdot \left(\frac{1}{g(x)},x,x,x\right).$$
In order to guarantee that this inverse exists for every element $(g, f_1, f_2, f_3)$, we have the following proposition.
\begin{proposition}
The determinant of the matrix 
$$\begin{pmatrix}
f_1^{(2)}(x^2 r_3(x)) & f_1^{(3)}(x^2 r_3(x)) \\
f_{12}^{(2)}(x^2 r_3(x))/x & f_{12}^{(3)}(x^2 r_3(x))/x
\end{pmatrix}$$ is non-zero.
\end{proposition}
\begin{proof}
We have $f_1, f_2 \in \mathcal{F}_1$, hence $f_1f_2\in \mathcal{F}_2$. Then $f_1^(2) \in \mathcal{F}_0$ and $(f_1f_2)^{(3)} \in \mathcal{F}_0$. Also $(f_1f_2)^{(2)} \in \mathcal{F}_1$. This now ensures that 
$$f_1^{(2)}(x^2r_3(x))(f_1f_2)^{(3)}(x^2r_3(x))-f_1^{(3)}(x^2r_3(x))(f_1f_2)^{(2)}(x^2r_3(x)) \neq 0.$$ 
Thus the determinant is non-zero.
\end{proof}
\begin{example}
We let $$(g_1,f_1,f_2,f_3)=\left(\frac{1}{1-x}, x(1+x), \frac{x}{1-3x}, \frac{x}{1-x^3}\right).$$
Then the matrix whose $(n,k)-$th term is given by
$$t_{n,k}=[x^n]g(x)f_1(x)^{\frac{2}{3}\left(1-\cos\left(\frac{2 \pi k}{3}\right)\right)}f_2(x)^{\frac{2}{3}\left(\frac{1}{2}+\cos\left(\frac{2 \pi (k+1)}{3}\right)\right)}(x^2f_3(x))^{\lfloor \frac{k}{3} \rfloor},$$
begins
$$\left(
\begin{array}{ccccccc}
1 &0 & 0 & 0 & 0 & 0 & 0 \\
1 &1 & 0 & 0 & 0 & 0 & 0 \\
1 &2 & 1 & 0 & 0 & 0 & 0 \\
1 & 2 & 5 & 1 & 0 & 0 & 0 \\
1 & 2 & 17 & 1 & 1 & 0 & 0 \\
1 & 2 & 53 & 1 & 2 & 1 & 0 \\
1 & 2 & 161 & 2 & 2 & 5 & 1 \\
\end{array}
\right).$$
The production array of this matrix will then begin
$$\left(
\begin{array}{ccccccc}
1 &1 & 0 & 0 & 0 & 0 & 0 \\
0 &1 & 1 & 0 & 0 & 0 & 0 \\
0 &-1 & 3 & 1 & 0 & 0 & 0 \\
0 & 4 & 0 & -4 & 1 & 0 & 0 \\
0 & 12 & 0 & 12 & 1 & 1 & 0 \\
0 & 24 & 0 & -23 & -1 & 3 & 1 \\
0 & 8 & 0 & -12 & 4 & 0 & -4 \\
\end{array}
\right), $$ showing three vertical stripes after the first (Z) column. The production array of the inverse of this matrix begins
$$\left(
\begin{array}{ccccccc}
-1 &1 & 0 & 0 & 0 & 0 & 0 \\
0 &-1 & 1 & 0 & 0 & 0 & 0 \\
-3 &5 & -3 & 1 & 0 & 0 & 0 \\
-31 & 61 & -35 & 4 & 1 & 0 & 0 \\
-103 & 205 & -119 & 17 & -1 & 1 & 0 \\
-279 & 557 & -330 & 49 & 5 & -3 & 1 \\
-779 & 1557 & -934 & 125 & 61 & -35 & 4 \\
\end{array}
\right).$$
The triple striped nature of these production matrices is clear.

In this case, we have
$$f_1^{(1)}(x)=0, f_1^{(2)}=1, f_1^{(3)}=1,$$ and 
$$f_{12}^{(1)}(x)=\frac{4}{1-27x}, f_{12}^{(2)}(x)=\frac{12x}{1-27x}, f_{12}^{(3)}(x)=\frac{1+9x}{1-27x}.$$ 
With these values, we obtain 
$$r_1(x)=\frac{x(1-x+4x^2+10x^3+26x^4)}{1-2x^3},$$ 
$$r_2(x)=\frac{x(1-4x-12x^2-26x^3)}{1-x+4x^2+10x^3+26x^4},$$ and 
$$r_3(x)=\frac{x}{1+x^3}.$$ 
This gives 
$$(1,r_1(x), r_2(x), r_3(x))=(1,f_1(x),f_2(x),f_3(x))^{-1},$$ 
and so
\begin{align*}(g(x),f_1(x),f_2(x),f_3(x))^{-1}&=(1,r_1(x), r_2(x), r_3(x))\cdot \left(\frac{1}{g(x)},x,x,x\right)\\
&=\left((1,r_1,r_2,r_3)\cdot(1/g), \frac{(1,r_1,r_2,r_3)\cdot(x/g)}{(1,r_1,r_2,r_3)\cdot(1/g)}, \frac{(1,r_1,r_2,r_3)\cdot(x^2/g)}{(1,r_1,r_2,r_3)\cdot(x/g)},r_3(x)\right)\\
&=(w(x), s_1(x),s_2(x),s_3(x))\end{align*} 
where 
\begin{align*}
w(x)&=\frac{1-x+x^2-6x^3-10x^4-26x^5}{1-2x^3}\\
s_1(x)&=\frac{x(1-2x+8x^2+22x^3+52x^4)}{1-x+x^2-6x^3+10x^4+26x^5}\\
s_2(x)&=\frac{x(1-5x-12x^2-25x^3-2x^4-12x^5-26x^6)}{(1+x^3)(1-2x+8x^2+22x^3+52x^4)}\\
s_3(x)&=\frac{x}{1+x^3}.\end{align*}
\end{example}
The general production matrix for this group on four elements will have the form
$$\left(\begin{array}{rrrrrrr}
z_0 & a_0 & 0 & 0 & 0 &  0 & \cdots \\
z_1 & a_1 & b_0 & 0 & 0 & 0 & \cdots \\
z_2 & a_2 & b_1 & c_0 & 0 & 0 & \cdots \\
z_3 & a_3 & b_2 & c_1 & a_0 & 0 & \cdots \\
z_4 & a_4 & b_3 & c_2 & a_1 & b_0& \cdots \\
z_5 & a_5 & b_4 & c_3 & a_2 & b_1 & \cdots \\
\vdots & \vdots & \vdots & \vdots & \vdots & \vdots & \ddots
\end{array}\right).$$ The distinguishing character of these production matrices is illustrated by aligning the $A$, $B$, and $C$ sequences of the above numeric example.
$$\left(
\begin{array}{ccccccccc}
A & B  & C   & A+B+C & | & A^* & B^* & C^* & A^*+B^*+C^*\\
1 &1   & 1   & 3     & | & 1    & 1     & 1     & 3 \\
1 &3   & -4  & 0     & | & -1   & -3    & 4     & 0\\
-1 &0  & 12  & 11    & | & 5    & -35   & 17    & -13\\
4 & 0  & -23 & -19   & | & 61   & -119  & 49    & -9\\
12 & 0 & -12 & 0     & | & 205  & -350  & 125   & 0\\
24 & 0 & -36 & -12   & | & 557  & -934  & 365   & -12\\
8 & 0  & -72 & -64   & | & 1557 & -2710 & 1125  & -28\\
24 & 0 & -24 & 0     & | & 4485 & -7918 & 3433  & 0\\
48 & 0 & -72 & -24   & | & 13029 & -23458 & 10393 & -36\\
\end{array}
\right).$$ We see an emerging pattern of zeros repeating after every third number. Matrices from this Sprugnoli group on four elements conform to the following schema.
$$
\begin{array}{ccccccccc}
g &g & g & g & g & g & g & g & g  \\
1 &f_1 & f_1 & 1 & f_1 & f_1 & 1 & f_1 & f_1\\
1 &1 & f_2 & 1 & 1 & f_2 & 1 & 1 & f_2\\
1 & 1 & 1 & f_3 & f_3 & f_3 & f_3^2 & f_3^2 & f_3^2\\
1 & 1 & 1 & x^2 & x^2 & x^2 & x^4 & x^4 & x^4\\
- & - & - & - & - & - & - & - & -\\
0 & 1 & 2 & 3 & 4 & 5 & 6 & 7 & 8\\
\end{array}
$$
In a similar way, elements $(g,f_1,f_2,f_3,f_4)$ (with $f_4 \in xR[[x^4]]$) in the Sprugnoli group on five elements conform to the following schema.
$$
\begin{array}{ccccccccccc}
g &g & g & g & g & g & g & g & g & g & g \\
1 &f_1 & f_1 & f_1 & 1 & f_1 & f_1 & f_1 & 1& f_1 & f_1 \\
1 &1 & f_2 & f_2 & 1 & 1 & f_2 & f_2 & 1& 1 & f_2 \\
1 & 1 & 1 & f_3 & 1 & 1 & 1 & f_3 & 1& 1 & 1 \\
1 & 1 & 1 & 1 & f_4 & f_4 & f_4 & f_4 & f_4^2& f_4^2 & f_4^2 \\
1 & 1 & 1 & 1 & x^3 & x^3 & x^3 & x^3 &x^6& x^6 & x^6 \\
- & - & - & - & - & - & - & - & -& - & - \\
0 & 1 & 2 & 3 & 4 & 5 & 6 & 7 & 8& 9 & 10 \\
\end{array}
$$
In this case, the four quadrisections of a generating function $g(x)$ are given by
$$\frac{1}{4}
\begin{pmatrix}
1 & 0 & 0 & 0\\
0 & x^{-1/4} & 0 & 0\\
0 & 0 & x^{-2/4} & 0\\
0 & 0 & 0 & x^{-3/4}
\end{pmatrix}
\begin{pmatrix}
1 & 1 & 1 & 1 \\
1 & -i & -1 & i\\
1 & -1 & 1 & -1\\
1 & i & -1 & -i
\end{pmatrix}
\begin{pmatrix}
g\!\left(x^{1/4}\right) \\
g\!\left(i x^{1/4}\right)\\
g\!\left(- x^{1/4}\right)\\
g\!\left(-i x^{1/4}\right)
\end{pmatrix}
=
\begin{pmatrix}
g^{(1)}(x) \\
g^{(2)}(x) \\
g^{(3)}(x) \\
g^{(4)}(x)
\end{pmatrix}.$$
Setting $(1,r_1,r_2,r_3,r_4)=(1,f_1,f_2,f_3,f_4)^{-1}$, we solve for $r_1,r_2$ and $r_3$ using the matrix formulation
\begin{scriptsize}
$$\begin{pmatrix}
f_1^{(2)}(x^3 r_4(x)) & f_1^{(3)}(x^3 r_4(x)) & f_1^{(4)}(x^3r_4(x)) \\
(f_1f_2)^{(2)}(x^3 r_4(x))/x & (f_1f_2)^{(3)}(x^3 r_4(x))/x & (f_1f_2)^{(4)}(x^3 r_4(x))/x \\
(f_1f_2f_3)^{(2)}(x^3 r_4(x))/x^2 & (f_1f_2f_3)^{(3)}(x^3r_4(x)) & (f_1f_2f_3)^{(4)}(x^3r_4(x))
\end{pmatrix}
\begin{pmatrix}
r_1(x) \\
r_1(x)r_2(x)\\
r_1(x)r_2(x)r_3(x)
\end{pmatrix}
=
\begin{pmatrix}
x-f_1^{(1)}(x^3 r_4(x)) \\
x-(f_1f_2)^{(1)}(x^3 r_4(x))/x\\
x-(f_1f_2f_3)^{(1)}(x^3 r_4)/x^2
\end{pmatrix}.$$ 
\end{scriptsize}
Here, we have $$r_4(x)=\frac{1}{x^3}\left(\overline{\left(x^3f_4(x)\right)^{1/4}}\right)^4.$$
Continuing in this fashion indicates that we can construct general $n$-Sprugnoli groups on elements $(g,f_1, f_2,\ldots, f_n)$.
In particular, the product in the $n$-Sprugnoli group is given by 
\begin{scriptsize}
$$(g,f_1,\ldots,f_n)\cdot (h,r_1,\ldots,r_n)=$$
$$\left((g,f_1,\ldots,f_n)\cdot h, \frac{(g,f_1,\ldots,f_n)\cdot hr_1}{(g,f_1,\ldots,f_n)\cdot h},\ldots, \frac{(g,f_1,\ldots,f_n)\cdot hr_1\cdots r_{i+1}}{(g,f_1,\ldots,f_n)\cdot hr_1\cdots r_i},\ldots, \frac{1}{x^{n-1}}\frac{(g,f_1,\ldots,f_n)\cdot hx^{n-1} r_n}{(g,f_1,\ldots,f_n)\cdot h}\right).$$
\end{scriptsize}

\section{Further examples}
The following are examples of some simple Sprugnoli arrays and their inverses.
\begin{align*}\left(\frac{1}{1+x}, \frac{x}{1-x}, \frac{x}{1-x^2}\right)^{-1}&=\left(\frac{1+x}{1+x^2}, \frac{x}{1+x}, \frac{x}{1+x^2}\right)\\
\left(\frac{1}{1-x},x,x(1+x^2)\right)^{-1}&=\left(1-x, \frac{x(1-xc(-x^2))}{1-x}, xc(-x^2)\right)\end{align*}
Here, we have $c(x)=\frac{1-\sqrt{1-4x}}{2x}$, the generating function of the Catalan numbers.
\begin{align*}\left(\frac{1}{1-x}, x(1+x),x(1+x^2)\right)^{-1}&=(1-x+x^2c(-x^2), x(-1+(2-x)c(-x^2)), xc(-x^2))\\
\left(\frac{1}{1-2x},\frac{x(1+x)}{1-x},\frac{x}{1-x^2}\right)^{-1}&=\left(\frac{1-2x+6x^2}{1+2x^2}, \frac{x(1-4x-x^2-6x^3}{(1+x^2)(1-2x+6x^2)}, \frac{x}{1+x^2}\right)\\
\left(\frac{1}{1-x}, \frac{x(1+2x)}{1-x}, \frac{x}{1-x^2}\right)^{-1}&=\left(\frac{1-x+6x^2}{1+3x^2}, \frac{x(1-4x+x^2-6x^3)}{(1+x^2)(1-x+6x^2)}, \frac{x}{1+x^2}\right)\\
\left(\frac{1}{1-x}, \frac{x(1+x)}{1-x}, \frac{x}{1-x^2}\right)^{-1}&=\left(\frac{1-x+4x^2}{1+2x^2}, \frac{x(1-3x+x^2-4x^3)}{(1+x^2)(1-x+4x^2)}, \frac{x}{1+x^2}\right).\end{align*}
\begin{example} \textbf{An involution in the Sprugnoli group}. We have
$$\left(\frac{1}{1-x}, \frac{-x}{1+x}, \frac{-x}{1-x^2}\right)^{-1}=\left(\frac{1}{1-x}, \frac{-x}{1+x}, \frac{-x}{1-x^2}\right).$$
Thus this matrix $M$ satisfies $M^2=I$. This is an example of an involution in the Sprugnoli group. We note that the production matrix of this Sprugnoli array has the simple form as follows.
$$\left(
\begin{array}{ccccccc}
1 &-1 & 0 & 0 & 0 & 0 & 0 \\
0 &-1 & 1 & 0 & 0 & 0 & 0 \\
0 & 0 & 1 & -1 & 0 & 0 & 0 \\
0 & 0 & 0 & -1 & 1 & 0 & 0 \\
0 & 0 & 0 & 0 & 1 & -1 & 0 \\
0 & 0 & 0 & 0 & 0 & -1 & 1 \\
0 & 0 & 0 & 0 & 0 & 0 & 1 \\
\end{array}
\right).$$
\end{example}
\section{Conclusions and acknowledgements}
Matrix groups based on the idea of using generating functions to prescribe the columns of the matrices in the group have grown in popularity since the first definition of the Riordan group \cite{SGWW}. The combination of generating functions and matrices has led to many applications of these groups in combinatorics and elsewhere, where the idea of generalized weighted compositions comes into play. This note adds to this context by defining a new group of matrices whose columns are defined by generating functions in an original way. We have indicated how these ideas define a hierarchy of Sprugnoli groups with elements $(g,f_1,f_2, \ldots, f_n)$ where there is a specific prescription on the form of $f_n$. In the definition of the Sprugnoli group on elements $(g, f_1, f_2)$ above, there is a strong implicit role played by the (once-)stretched Riordan arrays $(g, xf_2)$ and $(gf_1/x, xf_2)$. Embedded in the inverse $(g, f_1, f_2)^{-1}$ we can find left-inverses \cite{LI} of these stretched Riordan arrays. Higher order Sprugnoli groups are based on further stretchings of base Riordan arrays. In this way, the Riordan group is a fundamental building block for the Sprugnoli family of groups.

The author would like to thank Donatella Merlini for carefully reading a first draft of this paper, and for making pertinent suggestions which I hope have led to a clearer and more complete picture. Professor Merlini was a close collaborator of Renzo Sprugnoli and I am most grateful for her input.
\section{Renzo Sprugnoli (1942-2024)}
Following the publication of the seminal paper \cite{SGWW} on the Riordan group by Shapiro, Getu, Woan, and Woodson, Renzo Sprugnoli was one of the first people to exploit the expressive power of this new group. In a succession of papers \cite{Tree, Exclude, LI, SEQ, ID, Comp, Geo, Binary, Implicit, Inv, CurId, AC, book2, Bib, Gould, Abel, CS}, and through his collaboration with others, he did much to show the power of these new methods, notably in the area of combinatorial sums and lattice paths. His fine analytical skills and his willingness to share with colleagues did much to further research in the area of Riordan arrays and their applications.

\section{Appendix: Even and odd bisections}
We let $g(x)$ by a power series $g(x)=\sum_{n=0}^{\infty} a_n x^n$. The even bisection of $g(x)$ is the power series
$$g^e(x)=\sum_{n=0}^{\infty} a_{2n} x^n$$ while the odd bisection of $g(x)$ is the series
$$g^o(x)=\sum_{n=0}^{\infty} a_{2n+1}x^n.$$
We then have
$$g^e(x)=\frac{g(\sqrt{x})+g(-\sqrt{x})}{2},$$
and
$$g^o(x)=\frac{g(\sqrt{x})-g(-\sqrt{x})}{2 \sqrt{x}}.$$
We can write this in matrix form as 
$$\frac{1}{2}
\begin{pmatrix}
1 & 0 \\
0 & x^{-1/2}
\end{pmatrix}
\begin{pmatrix}
1 & 1 \\
1 & -1
\end{pmatrix}
\begin{pmatrix}
g\!\left(x^{1/2}\right) \\
g\!\left(-x^{1/2}\right)
\end{pmatrix}
=
\begin{pmatrix}
g^e(x) \\
g^o(x)
\end{pmatrix}.$$
\begin{example} Given a power series $g(x)$, we calculate the even and odd bisections of $xg(x)$. We have
\begin{align*}
(xg)^e(x)&=\frac{(xg)(\sqrt{x})+(xg)(-\sqrt{x})}{2}\\
&=\frac{\sqrt{x}g(\sqrt{x})-\sqrt{x}g(-\sqrt{x})}{2}\\
&=\frac{\sqrt{x}(g(\sqrt{x})-g(-\sqrt{x}))}{2}\\
&=\frac{x(g(\sqrt{x})-g(-\sqrt{x}))}{2\sqrt{x}}\\
&=xg^o(x).\end{align*}
Thus $$(xg)^e(x)=xg^o(x).$$
Similarly, we have
\begin{align*}
(xg)^o(x)&=\frac{(xg)(\sqrt{x})-(xg)(-\sqrt{x})}{2\sqrt{x}}\\
&=\frac{\sqrt{x}g(\sqrt{x})+\sqrt{x}g(-\sqrt{x})}{2\sqrt{x}}\\
&=\frac{g(\sqrt{x})+g(-\sqrt{x})}{2}\\
&=g^e(x).\end{align*}
Thus $$(xg)^o(x)=g^e(x).$$
\end{example}

\bigskip
\hrule
\bigskip
\noindent 2020 {\it Mathematics Subject Classification}:
Primary 15B36; Secondary 05A15, 11B83, 15A30, 20H25.

\noindent \emph{Keywords:} Matrix group, Riordan group, double Riordan group, generating function, production matrix.

\bigskip
\hrule
\bigskip
\noindent (Concerned with sequence
\seqnum{A000045},
\seqnum{A000108},
\seqnum{A051159},
\seqnum{A055879}, and
\seqnum{A122367}
).

\end{document}